\documentclass{article}

\usepackage[english]{babel}

\usepackage[a4paper,top=2cm,bottom=2cm,left=3cm,right=3cm,marginparwidth=1.75cm]{geometry}

\usepackage[colorlinks=true, allcolors=blue]{hyperref}
\usepackage{amssymb, amscd, amsthm, amsmath,latexsym, amstext,mathrsfs,multicol,CJK, mathtools}
\usepackage[all]{xy}
\usepackage[dvips]{graphicx}
\usepackage{relsize}

\newtheorem{thm}{Theorem}[section]
\newtheorem{coro}[thm]{Corollary}
\newtheorem{lemma}[thm]{Lemma}
\newtheorem{prop}[thm]{Proposition}

\theoremstyle{definition}
\newtheorem{defn}[thm]{Definition}

\newtheorem{remark}[thm]{Remark}

\numberwithin{equation}{thm}
\newcommand{\Gal}{\mathrm{Gal}}

\title{Integral embeddings of central simple algebras over number fields}

\author{Jiaqi Xie and Fei Xu}

\begin{document}
	\maketitle
	
	\begin{abstract} 
	Let $\mathcal A$ be a central simple algebra of degree $n\geq 3$ over a number field $k$ and $K/k$ be a field extension of degree $n$ inside $\mathcal A$. A  criterion for determining exactly when an order of $K$ can be embedded in an order of $\mathcal A$ is given. Various previous results have been generalized and recovered by applying this criterion. 	
	\end{abstract}

	
	\section{Introduction} \

As a consequence of class field theory, a finite extension $K/k$ of number fields of degree $n$ can be embedded into a central simple algebra $\mathcal A$ of degree $n$ over $k$ if and only if $K\otimes_k k_v$ can be embedded into $A\otimes_k k_v$ over $k_v$ for all primes $v$ of $k$ (see \cite[Proposition A.1]{PR}).  Suppose that $\mathcal A\otimes_k k_v= M_{m_v} (D_v)$ where $D_v$ is a central division algebra of degree $d_v$ for each prime $v$ of $k$ and $\mathfrak v$'s are the  primes of $K$ above $v$. Then $K\otimes_k k_v$ can be embedded into $A\otimes_k k_v$ if and only if $d_v | [K_{\mathfrak v} : k_v]$ for all $\mathfrak v$ above $v$ (see \cite[Proposition A.1]{PR}).

However, the local-global principle of integral embeddings of orders does not hold any more. Indeed, Chevalley in \cite[Th\'eor\`eme 10]{Ch} first pointed out that 
$$ \frac{\sharp \{\text{conjugacy classes of maximal orders in $M_n(k)$ where $\mathfrak{o}_K$ can be embedded} \} }{\sharp \{\text{conjugacy classes of maximal orders in $M_n(k)$} \}}= \frac{1}{[K \cap H_k : k]}$$
where $\mathfrak {o}_K$ is the ring of integers of $K$ and $H_k$ is the Hilbert class field of $k$. It is natural to investigate this problem for general central simple algebras. Indeed, Chinburg and Friedman studied such embeddings of maximal orders for quaternion algebras and established the Chevalley's result for quaternion algebras in \cite{CF}. Several extensions to other kinds of orders of quaternions have been given in \cite{CX}, \cite{GQ} and \cite{Lin}. The integral embeddings of orders for certain central simple algebras of higher degree was studied in \cite{Are}, \cite{Are0}, \cite{Are1}, \cite{LS} and \cite{LS1}. Since existence of integral embeddings of orders for central simple algebras is equivalent to existence of integral points of the corresponding homogeneous spaces, we study this issue by applying strong approximation with Brauer-Manin obstruction for homogeneous spaces in \cite{ref5}.  The main result of this paper is the following theorem (see Theorem \ref{rec}).

\begin{thm}\label{main} Let $\mathcal A$ be a central simple algebra of degree $n\geq 3$ over a number field $k$ and $K/k$ be a field extension of degree $n$ inside $\mathcal A$. Suppose that
$\Xi$ is an $\mathfrak o_k$-order of $K$. Then $\Xi$ can be embedded into an $\mathfrak o_k$-order $\Gamma$ of $\mathcal A$ as $\mathfrak o_k$-algebras if and only if  there is 
$g_v\in (\mathcal A\otimes_k k_v)^\times$ for every prime $v$ of $k$ with  $$\Xi \otimes_{\mathfrak o_{k}} \mathfrak o_{k_v}  \subset g_v^{-1} (\Gamma \otimes_{\mathfrak o_k} \mathfrak o_{k_v} ) g_v  \ \ \ \text{ for all $v< \infty_k$ } $$ and $g_v\in (\Gamma \otimes_{\mathfrak o_k} \mathfrak o_{k_v} )^\times$ for almost all $v$ such that 
$$ (N_v(g_v))_v \in k^\times \cdot N_{K/k} (\mathbb I_K )  $$
where $N_v$ is the reduced norm map from $\mathcal A\otimes_k k_v$ to $k_v$ and $\mathbb I_K$ is the idelic group of $K$. 
 \end{thm}

It should be pointed out that we do not touch optimal embeddings for central simple algebras of higher degree yet which were studied extensively for quaternions in \cite{Ma}, \cite{AAC}, \cite{Voi}, \cite{PX} and \cite{XY}.

Notation and Terminology are standard. Let $k$ be a number field, $\frak o_k$ be the ring of integers of $k$ and $\Omega_k$ be the set of all primes of $k$. For any $v\in \Omega_k$, we use $k_v$ to denote the completion of $k$ at $v$. The set of Archimedean primes of $k$ is denoted by $\infty_k$ and  a finite prime $v$ of $k$ is denoted by $v<\infty_k$. The adeles and ideles of $k$ are denoted by $\mathbb A_k$ and $\mathbb I_k$ respectively. We also use $\mathbb A_k^f$ and $\mathbb I_k^f$ to denote the finite adeles and the finite ideles of $k$ (without $\infty_k$-components) respectively. For any $v<\infty_k$ and a finite extension  $E/k_v$, the ring of integers of $E$ is denoted by $\frak o_E$. For any ring $R$, the group of units of $R$ is denoted by $R^\times$.  For a central simple algebra $\mathcal{A}$ of degree $n$ over $k$, we use $G_{\mathcal A}$ to denote a linear algebraic over $k$ which represents the functor
$$ ((  \text{commutative $k$-algebras})) \rightarrow ((\text{groups})); \ \ E \mapsto (E\otimes_k A)^\times  $$
and $SG_{\mathcal A}$ to be a closed subgroup of $G_{\mathcal A}$ over $k$ which represents the functor
$$ (( \text{commutative $k$-algebras})) \rightarrow ((\text{groups})); \ \ E \mapsto \ker ((E\otimes_{k} A)^\times\xrightarrow{N} E^\times)   $$
where $N$ is the reduced norm map. Then one has the exact sequence of linear algebraic groups
 \begin{equation}\label{exact}  1\longrightarrow SG_{\mathcal A} \longrightarrow G_{\mathcal A} \xrightarrow{ \ N \ } \mathbb G_m \longrightarrow 1  \end{equation} 
where $\mathbb G_m$ is the multiplicative group over $k$. For a variety $X$ over $k$, we write $\mathrm{Br}(X)=\mathrm{H}^2_{et}(X, \mathbb G_m)$ and the algebraic Brauer group of $X$ 
$$ \mathrm{Br}_1(X) = \ker (\mathrm{H}^2(X, \mathbb G_m) \rightarrow \mathrm{H}^2(X \times_k \bar k, \mathbb G_m)) $$
where $\bar k$ is a fixed algebraic closure of $k$.

\section{Class fields associated to orders}\

In this section, we list some basic arithmetic properties of central simple algebras over a number field which we need in the following section.  Some results have already been given in \cite{LS1}. We repeat them here for completeness.

\begin{lemma} \label{integral-norm} If $\frak o_{\mathcal A_v}$ is a  maximal order of a central simple algebra $\mathcal A_v$ over $k_v$ for a finite prime $v$ of $k$, then $N_v(\frak o_{\mathcal A_v}^\times)=\frak o_{k_v}^\times$ where $N_v$ is the reduced norm map of $\mathcal A_v$. 
\end{lemma}
\begin{proof} Since $\mathcal A_v \cong M_m(D_v)$ for some positive integer $m$ where $D_v$ is a central division algebra of degree $d$ over $k_v$ by \cite[Chapter IX, \S1, Theorem 1]{Weil},  
one has that $\frak o_{\mathcal A_v}$ is conjugate to $M_m(\Delta_v)$ where $\Delta_v$ is the maximal order of $D_v$ by \cite[(17.3) Theorem (i) and (ii)]{Rein}. By \cite[Chapter IX, \S2, Corollary 2 of Proposition 6]{Weil}, one only needs to prove that $N_0(\Delta_v^\times)=\frak o_{k_v}^\times$ where $N_0$ is the reduced norm of $D_v$.  By \cite[Chapter I, \S 4, Proposition 5 and Chapter IX, \S 4, Proposition 11]{Weil}, there is an unramified extension $F/k_v$ of degree $d$ inside $D_v$ such that  $$N_{F/k_v}(x) = N_0(x); \ \forall x\in F . $$   Since $\mathfrak o_F$ is contained in $\Delta_v$ and $N_{F/k_v}(\mathfrak o_F^\times)= \frak o_{k_v}^\times $ by \cite[Chapter VIII, \S1, Prop. 3]{Weil}, one concludes $N_0(\Delta_v^\times)=\frak o_{k_v}^\times$ as desired. 
\end{proof} 

\begin{lemma} \label{stab} Let $\mathcal A_v = M_m(D_v)$ be a central simple algebra over $k_v$ where $D_v$ is a central division algebra over $k_v$ for $v< \infty_k$. If $\Delta_v$ is the unique maximal order of $D_v$, then
$$ \{ g\in \mathcal A_v^\times : \ g M_{m} (\Delta_v) g^{-1}  = M_{m} (\Delta_v) \} = D_v^\times \cdot  M_{m} (\Delta_v)^\times \ \ \ \text{and} \ \ \ N_v(  D_v^\times \cdot  M_{m} (\Delta_v)^\times) = \frak o_{k_v}^\times (k_v^\times)^m $$ where $N_v$ is the reduced norm map of $\mathcal A_v$.
\end{lemma} 

\begin{proof}  For any $g\in \mathcal A_v^\times$ satisfying $g M_{m} (\Delta_v) g^{-1}  = M_{m} (\Delta_v)$, there are $S_1, S_2\in M_{m} (\Delta_v)^\times$ such that 
$$g = S_1  \cdot diag(\alpha_1, \cdots, \alpha_{m} ) \cdot S_2  \ \ \ \text{with} \ \ \ ord_v(\alpha_1)\leq \cdots \leq ord_v(\alpha_{m}) $$  by \cite[(17.7) Theorem]{Rein}, where $ord_v$ is the extension of $v$-adic valuation of $k_v$ over $D_v$. Since $$S_1^{-1} \cdot M_{m} (\Delta_v) \cdot S_1  = M_{m} (\Delta_v) \ \ \ \text{and} \ \ \ S_2 \cdot  M_{m} (\Delta_v) \cdot S_2^{-1}  = M_{m} (\Delta_v) , $$ one obtains $$ diag(\alpha_1, \cdots, \alpha_{m} ) \cdot  M_{m} (\Delta_v) \cdot diag(\alpha_1, \cdots, \alpha_{m} )^{-1}  = M_{m} (\Delta_v) . $$
This implies that $ord_v(\alpha_1)= \cdots = ord_v(\alpha_{m})$. Therefore 
$$ diag(\alpha_1, \cdots, \alpha_{m_p} )= \alpha_1 \cdot diag(1, \alpha_{1}^{-1} \alpha_2, \cdots, \alpha_1^{-1} \alpha_{m}) \in D_v^\times \cdot  M_{m} (\Delta_v)^\times . $$ 
Since $xyx^{-1}\in M_{m} (\Delta_v)^\times$ for $x\in D_v^\times$ and $y\in M_{m} (\Delta_v)^\times$, one has  $g\in D_v^\times \cdot  M_{m} (\Delta_v)^\times $ as desired. 

Since $$N_v(diag (\alpha, \cdots, \alpha))=N_0(\alpha)^m \ \ \ \text{ for} \ \  diag (\alpha, \cdots, \alpha)\in  M_m(D_v) \ \ \text{ with } \ \ \alpha\in D_v  $$
where $N_0$ is the reduced norm of $D_v$ over $k_v$, one has 
$$N_v(  D_v^\times \cdot  M_{m} (\Delta_v)^\times)=N_0(D_v^\times )^m \cdot N(M_{m} (\Delta_v)^\times)  = \frak o_{k_v}^\times (k_v^\times)^m $$ as desired by Lemma \ref{integral-norm} and 
\cite[(33.1) Theorem and (33.4) Theorem]{Rein}. 
\end{proof}

\begin{defn} Let $\mathcal A$ be a central simple algebra of degree $n$ over a number field $k$ and $\Gamma$ be a $\mathfrak o_k$-order of $\mathcal A$. Since $G_{\mathcal A}(\mathbb A_k)$ and $G_{\mathcal A}(k)$ act on $\Gamma$ by conjugation, one can define the genus of $\Gamma$ denoted by
$$ gen(\Gamma) = \text{the orbit of $\Gamma$ under the action of $G_{\mathcal A}(\mathbb A_k)$}  $$
and the class of $\Gamma$ denoted by 
$$ cls(\Gamma) = \text{the orbit of $\Gamma$ under the action of $G_{\mathcal A}(k)$}  . $$
\end{defn}

Since $G_{\mathcal A}(k) \subset G_{\mathcal A}(\mathbb A_k)$, one has $ cls(\Gamma)  \subseteq   gen(\Gamma)$. Moreover, the classes in $gen(\Gamma)$ are parametrized by the double cosets $G_{\mathcal A}(k) \backslash G_{\mathcal A}(\mathbb A_k)/  St_{G_{\mathcal A}}(\Gamma)$ where 
\begin{equation} \label{st}   
 St_{G_{\mathcal A}}(\Gamma) =\{ (g_v)\in G_{\mathcal A}(\mathbb A_k): \ g_v( \Gamma \otimes_{\frak o_k} \frak o_{k_v}) g_v^{-1}=\Gamma \otimes_{\frak o_k} \frak o_{k_v} \ \text{for all $v<\infty_k$} \} .  \end{equation}
By Lemma \ref{stab} and \cite[Chapter 8, Theorem 8.1]{ref7}, the number of classes in $gen(\Gamma)$ is finite. 

When $\Gamma$ is a maximal order of $\mathcal A$ over $k$, then 
$$ gen(\Gamma)= \{\text{all maximal orders of $\mathcal A$ over $k$}\} \ \ \ \text{and} \ \ \ cls(\Gamma)= \{ \text{the conjugacy classes of $\Gamma$ over $k$} \} $$
by \cite[(17.3) Theorem (i) and (ii)]{Rein}. 
 
 
 \begin{defn}  A central simple algebra $\mathcal A$ of degree $n$ over a number field $k$ satisfies the Eichler condition if $n\geq 3$ or $n=2$ and there is $v\in \infty_k$ such that $\mathcal A \otimes_k k_v \cong M_2(k_v)$. 
  \end{defn}
 
The significance of central simple algebras with the Eichler condition is that the simple and simply connected algebraic group $SG_{\mathcal A}$ over $k$ satisfies strong approximation.
The following proposition has been proved in \cite[Theorem 3.3]{Lin} for quaternions and \cite[Theorem 3.1]{LS1} for maximal orders.  

\begin{prop}\label{parametrized}  Let $\mathcal A$ be a central simple algebra of degree $n$ over a number field $k$ satisfying the Eichler condition. If $\Gamma$ is a $\mathfrak o_k$-order of $\mathcal A$, then the map induced by the reduced norm $N$ 
$$ G_{\mathcal A}(k) \backslash G_{\mathcal A}(\mathbb A_k)/  St_{G_{\mathcal A}}(\Gamma) \cong \mathbb I_k/ (k^\times \cdot N(St_{G_{\mathcal A}}(\Gamma))) $$
is bijective.
\end{prop}
\begin{proof}
Since $\mathcal A$ satisfies the Eichler condition, the simple and simply connected algebraic group $SG_{\mathcal A}$ over $k$ satisfies strong approximation by \cite[Chapter 7, Theorem 7.12]{ref7}.  Therefore $$ SG_{\mathcal A} (\mathbb A_k)= SG_{\mathcal A} (k) \cdot ( SG_{\mathcal A} (\mathbb A_k)\cap  St_{G_{\mathcal A}}(\Gamma)) . $$  
This implies that
 $$   G_{\mathcal A}(\mathbb A_k) / (G_{\mathcal A}(k) \cdot  SG_{\mathcal A} (\mathbb A_k) \cdot St_{G_{\mathcal A}}(\Gamma) ) \cong G_{\mathcal A}(k) \backslash G_{\mathcal A}(\mathbb A_k)/  St_{G_{\mathcal A}}(\Gamma)  . $$

By applying Galois cohomology for the short exact sequence (\ref{exact}) and Lemma \ref{integral-norm}, one has the following commutative diagram of  exact sequences
\[
  \xymatrix{
1 \ar[r] & SG_{\mathcal A}(k)  \ar[r] \ar[d] & G_{\mathcal A}(k)   \ar[r]^{\ \ N}  \ar[d]  &  k^\times \ar[r] \ar[d] & \mathrm{H}^1(k, SG_{\mathcal A})  \ar[d]^{\cong} \\
1 \ar[r] & SG_{\mathcal A} (\mathbb A_k) \ar[r] &   G_{\mathcal A}(\mathbb A_k)  \ar[r]^{ \ \ N} & \mathbb I_k \ar[r] & \prod_{v\in \infty_k} \mathrm{H}^1(k_v, SG_{\mathcal A})  
}\]
where the vertical maps are diagonal and the last one is bijective by \cite[Chapter 6, Theorem 6.6]{ref7}.  Then the reduced norm map $N$ induces an isomorphism  
$$  G_{\mathcal A}(\mathbb A_k) / (G_{\mathcal A}(k) \cdot  SG_{\mathcal A} (\mathbb A_k) \cdot St_{G_{\mathcal A}}(\Gamma) ) \cong 
 N (G_{\mathcal A}(\mathbb A_k)) / (N(G_{\mathcal A}(k)) \cdot N( St_{G_{\mathcal A}}(\Gamma)) ) . $$
Let 
 $$T=\{ v\in \infty_k:  \mathcal A\otimes_k k_v \cong M_{m_v} (\mathbb H) \ \text{where $\mathbb H$ is a definite quaternion algebra over $\mathbb R$}\} . $$ 
Then
\begin{equation} \label{idelic-g}  N (G_{\mathcal A}(\mathbb A_k))= \prod_{v\in \infty_k \setminus T} k_v^\times \times \prod_{v\in T} (k_v^\times)^2  \times \mathbb I_k^f \end{equation}
by Lemma \ref{integral-norm} and \cite[(33.1) Theorem and (33.4) Theorem]{Rein} and 
\begin{equation} \label{gn}  N(G_{\mathcal A}(k)) = \{ \alpha\in k^\times: \ \alpha >0 \ \text{in $k_v$ for} \  v\in T \} \end{equation}
by \cite[(33.15) Theorem]{Rein}. Since  $k^\times$ is dense inside $\prod_{v\in T} k_v^\times$, one concludes 
$$ N (G_{\mathcal A}(\mathbb A_k)) / (N(G_{\mathcal A}(k)) \cdot N( St_{G_{\mathcal A}}(\Gamma)) ) \cong  \mathbb I_k / (k^\times \cdot  N( St_{G_{\mathcal A}}(\Gamma) )) $$
as desired. 
\end{proof}

Now we extend the class field associated maximal orders of a central simple algebra in \cite[\S 3.1]{LS1} to arbitrary orders. 

\begin{defn} \label{cls-field} Let $\mathcal A$ be a central simple algebra of degree $n$ over a number field $k$ and $\Gamma$ be a $\mathfrak o_k$-order of $\mathcal A$. The finite abelian extension $k(\Gamma)/k$ satisfying 
$$\Gal(k(\Gamma)/k) \cong  \mathbb I_k/( k^\times \cdot   N( St_{G_{\mathcal A}}(\Gamma)) ) $$ by the Artin map in the class field theory is called the class field associated to $\Gamma$. 
\end{defn}

It is clear that $k(\Gamma)/k$ only depends on $gen(\Gamma)$. Here are some basic properties of $k(\Gamma)/k$.

\begin{prop} \label{property} Let $\mathcal A$ be a central simple algebra of degree $n$ over a number field $k$ and $\Gamma$ be a $\mathfrak o_k$-order of $\mathcal A$.

1)  If $\Gamma\otimes_{\mathfrak o_k} \mathfrak o_{k_v}$ is a maximal order of $\mathcal A\otimes_{k} k_v$ for $v< \infty_k$, then the prime $v$ is unramified in $k(\Gamma)/k$.

2) If $\mathcal A \otimes_k k_v$ is a division algebra and $\Gamma\otimes_{\mathfrak o_k} \mathfrak o_{k_v}$ is a maximal order of $\mathcal A\otimes_{k} k_v$ for $v< \infty_k$, then the prime $v$ splits completely in $k(\Gamma)/k$. 

3) If $\Gamma$ is a maximal order of $\mathcal A$, then $k(\Gamma)/k$ is independent of choice of maximal orders of $\mathcal A$. In this case, we simply write $k(\mathcal A)$ for $k(\Gamma)$. 

4) If $\mathcal A=M_n(k)$ and $\Gamma$ is a maximal order of $\mathcal A$, then $k(\mathcal A)/k$ is a subfield of the Hilbert class field $H_k$ of $k$ given by 
$$ \Gal(k(\mathcal A)/k) \cong \Gal(H_k/k) / \Gal(H_k/k)^n \cong Cl(k)/Cl(k)^n $$ 
where $Cl(k)$ is the ideal class group of $k$ and the second canonical isomorphism is induced by the Artin map.  
\end{prop}

\begin{proof}  1)  Since $\Gamma\otimes_{\mathfrak o_k} \mathfrak o_{k_v}$ is a maximal order of $\mathcal A\otimes_{k} k_v$, one has $\mathfrak o_{k_v}^\times \subset N( St_{G_{\mathcal A}}(\Gamma))$ by Lemma \ref{stab}. The result follows from \cite[Chapter XIII, \S 10, Corollary 3]{Weil}. 

  2)  Since $\mathcal A \otimes_k k_v$ is a division algebra and $\Gamma\otimes_{\mathfrak o_k} \mathfrak o_{k_v}$ is a maximal order of $\mathcal A\otimes_{k} k_v$, 
  one has $k_v^\times \subset N( St_{G_{\mathcal A}}(\Gamma))$ by Lemma \ref{stab}. The result follows from \cite[Chapter XIII, \S 10, Corollary 3]{Weil}.
  
  3)  It follows from that all maximal orders of $\mathcal A$ are in the same genus by \cite[(17.3) Theorem (i) and (ii)]{Rein}.  

  4)  Since $$N( St_{G_{\mathcal A}}(\Gamma))= (\prod_{v\in \infty_k} k_v^\times \times \prod_{v< \infty_k}  \mathfrak o_{k_v}^\times ) \cdot \mathbb I_k^n $$ by Lemma \ref{stab}, 
  the result follows from the canonical isomorphisms $$ \Gal(H_k/k) \cong Cl(k) \cong \mathbb I_k /( k^\times \cdot   (\prod_{v\in \infty_k} k_v^\times \times \prod_{v< \infty_k}  \mathfrak o_{k_v}^\times ) ) $$
by \cite[Chapter VI, \S 6, (6.9) Proposition]{Neu} and \cite[Chapter V, \S 3]{Weil}. 
\end{proof}

\section{Integral Brauer-Manin obstruction} \

Let $\mathcal A$ be a central simple algebra of degree $n$ over a number field $k$ and $K/k$ be a field extension of degree $n$ inside $\mathcal A$.  By Skolem-Noether Theorem (see \cite[(7.21) Theorem]{Rein}),   the following functor
$$ ((\text{commutative $k$-algebras})) \rightarrow ((\text{sets})); \ \ E \mapsto  \{ \sigma: E\otimes_k K  \rightarrow  E\otimes_k \mathcal A \ \  \text{$E$-algebra homomorphism} \} $$
is represented by a homogeneous space $X_{K, \mathcal A}$ over $k$ 
\begin{equation} \label{hom} X_{K, \mathcal A}\cong  G_{\mathcal A}/\mathrm{Res}_{K/k} (\mathbb G_m) \cong SG_{\mathcal A}/\Upsilon  \end{equation} 
where $\Upsilon =\ker (\mathrm{Res}_{K/k} (\mathbb G_m) \xrightarrow{N_{K/k}} \mathbb G_m)$ is the norm 1 torus. In fact,  one has the following commutative diagram of  short exact sequences

\begin{equation}\label{quotient}
  \begin{CD}
 @.   1  @.   1  \\ 
 @. @VVV @VVV \\
1  @>>>  \Upsilon @>>>  \mathrm{Res}_{K/k} (\mathbb G_m)  @>N_{K/k}>>    \mathbb G_m @>>> 1  \\
@. @VVV @ VVV @VVV \\
1 @>>> SG_{\mathcal A} @ >>>   G_{\mathcal A} @> N>>  \mathbb G_m @>>> 1 \\
@. @VVV @ VVV \\
@.  X_{K, \mathcal A} @ >{id}>>  X_{K, \mathcal A}  \\
@.  @VVV @ VVV \\
@.   1 @ . 1  
 \end{CD}
 \end{equation}

First we compute algebraic Brauer group of $X_{K, \mathcal A}$. 

\begin{prop}\label{br} Let $X_{K, \mathcal A}$ be defined by (\ref{hom}). Then the map in \cite[Proposition 2.12 (ii)]{ref5} induces a canonical isomorphism 
$$ \mathrm{Br}_1(X_{K, \mathcal A})/\mathrm{Br}(k) \cong \mathrm{H}^1(k, \widehat{\Upsilon}) \cong  \mathrm{Hom} (\Gal((k^{ab} \cap K)/k), \mathbb Q/\mathbb Z)$$
where $k^{ab}$ is the maximal abelian extension of $k$ inside a fixed algebraic closure of $k$ containing $K$. 
\end{prop}
\begin{proof}  By \cite[Proposition 2.12 (ii)]{ref5} and (\ref{hom}), there is a natural isomorphism 
$$ \mathrm{Br}_1(X_{K, \mathcal A})/\mathrm{Br}(k) \cong \mathrm{H}^1(k, \widehat{\Upsilon}) $$
where $\widehat{\Upsilon}$ is the character group of  $\Upsilon$. Since 
$$1\longrightarrow \Upsilon \longrightarrow \mathrm{Res}_{K/k} (\mathbb G_m) \xrightarrow{N_{K/k}} \mathbb G_m\longrightarrow 1 , $$ one obtains the short exact sequence 
$$ 0 \longrightarrow \mathbb Z \longrightarrow \widehat{\mathrm {Res}_{K/k} (\mathbb G_m)} \longrightarrow \widehat{\Upsilon} \longrightarrow 0 $$ of $\Gal (\bar k/k)$-module.  By Galois cohomology, one has 
$$  \mathrm{H}^1(k, \widehat{\mathrm {Res}_{K/k} (\mathbb G_m)}) \longrightarrow  \mathrm{H}^1(k, \widehat{\Upsilon}) \longrightarrow \mathrm{H}^2(k, \mathbb Z) \longrightarrow \mathrm{H}^2(k, \widehat{\mathrm {Res}_{K/k} (\mathbb G_m)}) . $$ Since $\widehat{\mathrm {Res}_{K/k} (\mathbb G_m)}= \mathrm{Ind}_{\Gal(\bar k/k)}^{\Gal(\bar k/K)} \mathbb Z$, one obtains 
$$  \mathrm{H}^1(k, \widehat{\mathrm {Res}_{K/k} (\mathbb G_m)})= \mathrm{H}^1(K, \mathbb Z) = \mathrm{Hom}_{ct} (\Gal(\bar k/K), \mathbb Z)=0  \ \ \ \text{and} \ \ \ \mathrm{H}^2(k, \widehat{\mathrm {Res}_{K/k} (\mathbb G_m)})= \mathrm{H}^2(K, \mathbb Z) $$
by Shapiro Lemma (see \cite[(1.6.4) Proposition]{ref6}).  Applying Galois cohomology to the following short exact sequence
$$ 0\longrightarrow \mathbb Z \longrightarrow \mathbb Q \longrightarrow \mathbb Q/\mathbb Z \longrightarrow 0  \ \ \ \text{with} \ \ \ \mathrm{H}^1(k, \mathbb Q)=\mathrm{H}^1(K, \mathbb Q)=\mathrm{H}^2(k, \mathbb Q)=\mathrm{H}^2(K, \mathbb Q)=0, $$ one obtains the following commutative diagram
\[
  \xymatrix{
 \mathrm{H}^1(k, \mathbb Q/\mathbb Z) \ar[r]^{\mathrm{Res}} \ar[d]_{\cong}  & \mathrm{H}^1(K, \mathbb Q/\mathbb Z)  \ar[d]^{\cong} \\
\mathrm{H}^2(k, \mathbb Z) \ar[r]^{\mathrm{Res}} &  \mathrm{H}^2(K, \mathbb Z)  
}\]
Therefore one concludes 
$$ \mathrm{H}^1(k, \widehat{\Upsilon}) \cong \ker ( \mathrm{H}^1(k, \mathbb Q/\mathbb Z) \xrightarrow{\mathrm{Res}} \mathrm{H}^1(K, \mathbb Q/\mathbb Z))$$
$$\cong \ker ( \mathrm{Hom}_{ct}(\Gal(\bar k/k, \mathbb Q/\mathbb Z) \xrightarrow{\mathrm{Res}} \mathrm{Hom}_{ct}(\Gal(\bar k/K, \mathbb Q/\mathbb Z))\cong  \mathrm{Hom} (\Gal((k^{ab} \cap K)/k), \mathbb Q/\mathbb Z)$$ as desired.
\end{proof}

We need the following crucial local result. 

\begin{lemma}\label{local-orbit}  Let $\mathcal A$ be a central simple algebra of degree $n$ over a number field $k$ and $K/k$ be a field extension of degree $n$ inside $\mathcal A$. Fix  a maximal order $\mathfrak o_{\mathcal A}$ of $\mathcal A$. If $\mathcal A\otimes_k k_v \cong M_n(k_v)$ and there is $g_0\in (\mathcal A\otimes_k k_v)^\times$ such that $$\mathfrak o_K\otimes_{\mathfrak o_k} \mathfrak o_{k_v} \subset g_0^{-1}( \mathfrak o_{\mathcal A} \otimes_{\mathfrak o_k} \mathfrak o_{k_v}) g_0 $$ for $v<\infty_k$, then 
$$ \{ g \in (\mathcal A\otimes_k k_v)^\times: \  \mathfrak o_K\otimes_{\mathfrak o_k} \mathfrak o_{k_v} \subset g^{-1} ( \mathfrak o_{\mathcal A} \otimes_{\mathfrak o_k} \mathfrak o_{k_v}) g \} = ( \mathfrak o_{\mathcal A} \otimes_{\mathfrak o_k} \mathfrak o_{k_v})^\times \cdot g_0 \cdot (K\otimes_k k_v)^\times .  $$
\end{lemma}
\begin{proof}  Write $\mathcal A\otimes_k k_v = M_n(k_v)$. By  \cite[(17.3) Theorem (i) and (ii)]{Rein}, there is $h\in (\mathcal A\otimes_k k_v)^\times$ such that $ \mathfrak o_{\mathcal A} \otimes_{\mathfrak o_k} \mathfrak o_{k_v}= h \cdot  M_n(\mathfrak o_{k_v}) \cdot h^{-1}$. Then 
$$  \{ g \in (\mathcal A\otimes_k k_v)^\times:   \mathfrak o_K\otimes_{\mathfrak o_k} \mathfrak o_{k_v} \subset g^{-1} ( \mathfrak o_{\mathcal A} \otimes_{\mathfrak o_k} \mathfrak o_{k_v}) g \} = h \cdot  \{ g \in (\mathcal A\otimes_k k_v)^\times:   \mathfrak o_K\otimes_{\mathfrak o_k} \mathfrak o_{k_v} \subset g^{-1} M_n ( \mathfrak o_{k_v}) g \} . $$
Without loss of generality, one can assume that $$\mathfrak o_K\otimes_{\mathfrak o_k} \mathfrak o_{k_v} \subset  \mathfrak o_{\mathcal A} \otimes_{\mathfrak o_k} \mathfrak o_{k_v}=M_n(\mathfrak o_{k_v})$$ with $g_0=1$.  Write $$K\otimes_k k_v= \bigoplus_{\mathfrak V | v} K_{\mathfrak V} \ \ \ \text{and} \ \ \  \mathfrak o_K\otimes_{\mathfrak o_k} \mathfrak o_{k_v} = \bigoplus_{\mathfrak V | v} \mathfrak o_{K_{\mathfrak V}}  . $$ Let $\{ \epsilon_1, \cdots, \epsilon_n \}$ be a basis of  $\mathfrak o_K\otimes_{\mathfrak o_k} \mathfrak o_{k_v}$ as $\mathfrak o_{k_v}$-module. For any 
$$ \xi \in \{ g \in (\mathcal A\otimes_k k_v)^\times:   \mathfrak o_K\otimes_{\mathfrak o_k} \mathfrak o_{k_v} \subset g^{-1} M_n ( \mathfrak o_{k_v}) g \} , $$
we choose $a\in \mathfrak o_{k_v}$ such that $a\cdot \xi^{-1}\in  M_n ( \mathfrak o_{k_v}) $ and consider a free $\mathfrak o_{k_v}$-module of rank $n$ inside $ \mathfrak o_K\otimes_{\mathfrak o_k} \mathfrak o_{k_v} $ with a basis given by 
$$(\eta_1, \cdots, \eta_n)= (\epsilon_1, \cdots, \epsilon_n) \cdot a \cdot \xi^{-1} . $$
Since 
$$ x\cdot (\eta_1, \cdots, \eta_n)= x \cdot (\epsilon_1, \cdots, \epsilon_n) \cdot a \cdot \xi^{-1} = (\epsilon_1, \cdots, \epsilon_n) \cdot a \cdot \xi^{-1} \cdot (\xi \cdot x \cdot \xi^{-1}) = (\eta_1, \cdots, \eta_n) \cdot (\xi \cdot x \cdot \xi^{-1}) $$
for any $x\in \mathfrak o_K\otimes_{\mathfrak o_k} \mathfrak o_{k_v}$ and $\xi \cdot x \cdot \xi^{-1} \in  M_n ( \mathfrak o_{k_v})$, one concludes that $\mathfrak o_{k_v} \eta_1+ \cdots +\mathfrak o_{k_v} \eta_n$ is an ideal of $\mathfrak o_K\otimes_{\mathfrak o_k} \mathfrak o_{k_v}$. Since $\mathfrak o_{K_{\mathfrak V}} $ is a discrete valuation ring for each $\mathfrak V| v$, one concludes that every ideal in $\mathfrak o_K\otimes_{\mathfrak o_k} \mathfrak o_{k_v}$ is principal. Then there is $b\in \mathfrak o_K\otimes_{\mathfrak o_k} \mathfrak o_{k_v}$ such that $$ \mathfrak o_{k_v} \eta_1+ \cdots +\mathfrak o_{k_v} \eta_n= b \cdot (\mathfrak o_K\otimes_{\mathfrak o_k} \mathfrak o_{k_v}) .$$ Since  $\{ b \epsilon_1, \cdots,  b \epsilon_n \}$ is also a basis of $\mathfrak o_{k_v} \eta_1+ \cdots +\mathfrak o_{k_v} \eta_n$, there is $u\in GL_n(\mathfrak o_{k_v})$ such that 
$$ (b \epsilon_1, \cdots,  b \epsilon_n)= (\eta_1, \cdots, \eta_n) u =  (\epsilon_1, \cdots, \epsilon_n) \cdot a \cdot \xi^{-1} u $$
and $b\in (K\otimes_k k_v)^\times$. Therefore $$\xi=u \cdot (ab^{-1}) \in  ( \mathfrak o_{\mathcal A} \otimes_{\mathfrak o_k} \mathfrak o_{k_v})^\times \cdot (K\otimes_k k_v)^\times $$ as desired. 
\end{proof} 


The main result of this paper is the following theorem.

\begin{thm} \label{rec}
Let $\mathcal A$ be a central simple algebra of degree $n$ over a number field $k$ satisfying the Eichler condition and $K/k$ be a field extension of degree $n$ inside $\mathcal A$.  Suppose that $\Xi$ is an $\mathfrak o_k$-order of $K$. Then $\Xi$ can be embedded into an $\mathfrak o_k$-order $\Gamma$ of $\mathcal A$ if and only if  there is 
$(g_v)_v\in G_{\mathcal A}(\mathbb A_k)$ with $$\Xi \otimes_{\mathfrak o_{k}} \mathfrak o_{k_v}  \subset g_v^{-1} (\Gamma \otimes_{\mathfrak o_k} \mathfrak o_{k_v} ) g_v  \ \ \ \text{ for all $v< \infty_k$ } $$ such that $\rho ((N_v(g_v))_v)$ is trivial in $\Gal((k^{ab}\cap K)/k)$, where 
$$ \rho: \ \Bbb I_k \longrightarrow \Gal((k^{ab}\cap K)/k)$$ is the Artin map and 
 $N_v$ is the reduced norm map from $\mathcal A\otimes_k k_v$ to $k_v$ for all $v\in \Omega_k$. Equivalently, 
 $$ (N_v(g_v))_v \in k^\times \cdot N_{K/k} (\mathbb I_K) . $$
\end{thm}	
	
\begin{proof}  Since $\mathrm{H}^1(k_v, \mathrm{Res}_{K/k}(\mathbb{G}_m))=0$ by Hilbert 90 (see \cite[(6.2.1) Theorem]{ref6}),  one can identify 
\begin{equation} \label{iden}  X_{K, \mathcal A}(k_v)\cong G_{\mathcal A}(k_v) / (K\otimes_k k_v)^\times  \end{equation}  by (\ref{hom}).  Let $$ C_v=\{ \sigma_v\in X_{K, \mathcal A}(k_v):  \sigma_v (\Xi \otimes_{\mathfrak o_{k}} \mathfrak o_{k_v}) \subset \Gamma \otimes_{\mathfrak o_k} \mathfrak o_{k_v} \} $$
for any $v<\infty_k$.  Under this identification (\ref{iden}), one has 
$$ C_v \cong \{ g_v \in (\mathcal A\otimes_k k_v)^\times : \  (\Xi \otimes_{\mathfrak o_{k}} \mathfrak o_{k_v}) \subset g_{v}^{-1} (\Gamma \otimes_{\mathfrak o_k} \mathfrak o_{k_v}) g_v \} / (K\otimes_k k_v)^\times  $$ for $v< \infty_k$ by Skolem-Noether Theorem (see \cite[(7.21) Theorem]{Rein}). 
By the assumption, one has $C_v\neq \emptyset$ and $C_v$ is open and compact for all $v<\infty_k$. Since both $\Xi \otimes_{\mathfrak o_{k}} \mathfrak o_{k_v}$ and  $\Gamma \otimes_{\mathfrak o_k} \mathfrak o_{k_v} $ are maximal orders for almost all primes $v$ by \cite[\S 10, Exercise 3]{Rein}, one has $$C_v \cong (\Gamma \otimes_{\mathfrak o_k} \mathfrak o_{k_v})^\times \cdot (K\otimes_k k_v)^\times $$ for almost all $v< \infty_k$ under the identification (\ref{iden}) by Lemma \ref{local-orbit}.

Since 
$$ X_{K, \mathcal A}(k) \cap ( \prod_{v\in \infty_k} X_{K, \mathcal A}(k_v) \times \prod_{v< \infty_k} C_v)  = \{\sigma: \Xi  \longrightarrow \Gamma \ \ \text{as $\mathfrak o_k$-algebra homomorphism} \} $$
by \cite[Lemma 9.4.6 and Lemma 9.5.3]{Voi},  one concludes that  $\Xi$ can be embedded into $\Gamma$ as $\mathfrak o_k$-algebra if and only if
\begin{equation} \label{iff}  ( \prod_{v\in \infty_k} X_{K, \mathcal A}(k_v) \times \prod_{v< \infty_k} C_v)^{\mathrm{Br}_1(X_{K, \mathcal A})} \neq \emptyset \end{equation} 
by \cite[Theorem 3.7]{ref5} and \cite[Corollary 2.9]{LX}.

Applying Galois cohomology to (\ref{quotient}), one obtains the following commutative diagram of exact sequences
 $$
 \begin{CD}
 @.   1  @.   1  \\ 
 @. @VVV @VVV \\
1  @>>>  \Upsilon (k_v) @>>>  (K\otimes_k k_v)^\times  @>N_{K/k}>>  k_v^\times @>{\gamma}>> \mathrm{H}^1(k_v, \Upsilon) @>>> 1 \\
@. @VVV @ VVV @VV{id}V \\
1 @>>> SG_{\mathcal A} (k_v) @ >>>   (\mathcal A\otimes_k k_v)^\times  @> N>>  k_v^\times @>>> 1 \\
@. @VVV @ VVV \\
@.  X_{K, \mathcal A} (k_v) @ >{id}>>  X_{K, \mathcal A}(k_v)  \\
@.  @V{ev}VV @ VVV \\
@.  \mathrm{H}^1(k_v, \Upsilon)  @ . 1  \\
@. @ VVV \\
@. 1 
 \end{CD}
$$
for $v\in \Omega_k$.  For any $g_v(K\otimes_k k_v)^\times \in X_{K, \mathcal A} (k_v)$ with $g_v\in  (\mathcal A\otimes_k k_v)^\times$, there is  
 $$b\in K\otimes_k \bar k_v\subset \mathcal A\otimes_k \bar k_v \ \ \ \text{such that} \ \ \ N(g_v) = N_{K/k} (b)  $$
 where $\bar k_v$ is an algebraic closure of $k_v$. Then $g_v^{-1} b \in SG_{\mathcal A} (\bar k_v)$ by \cite[\S 9, Exercise 1]{Rein} and 1-cocycle
 $$ \xi:  \ \Gal(\bar k_v/k_v) \longrightarrow \Upsilon(\bar k_v); \ \sigma \mapsto b^{-1}b^\sigma = (g_v^{-1} b)^{-1} (g_v^{-1} b)^{\sigma} $$
in $\mathrm{H}^1(k_v, \Upsilon)$ satisfies $ev(g_v(K\otimes_k k_v)^\times)=\gamma (N(g_v))=\xi $. This implies that the following diagram 
 \begin{equation}
\begin{CD}
  X_{K, \mathcal A} (k_v)  @ >{N}>>   {k_v}^\times / (N_{K/k} (K\otimes_k k_v)^\times)   \\
  @V{ev}VV  @V{\cong}V{\bar \gamma}V \\
     \mathrm{H}^1(k_v, \Upsilon) @ >>{id}>    \mathrm{H}^1(k_v, \Upsilon) 
  \end{CD}  \end{equation} 
commutes, where $\bar \gamma$ is induced by $\gamma$.  

For any $b\in \mathrm{Br}_1(X_{K, \mathcal A})$, there is $$\chi\in \mathrm{Hom} (\Gal((k^{ab} \cap K)/k), \mathbb Q/\mathbb Z) \cong \mathrm{H}^1(k_v, \widehat{\Upsilon}) \ \ \ \text{ and } \ \ \ \alpha\in \mathrm{Br}(k)$$ such that $\theta(\chi)= b+ \alpha$ by Proposition \ref{br}, where $\theta$ is the map defined in \cite[Proposition 2.12 (ii)]{ref5}.  
Since the following diagram of pairings 
 $$\begin{CD}
           X_{K, \mathcal A}(k_v)  @.  \times  \ \ \mathrm{Br}_1(X_{K, \mathcal A}\times_k k_v)  @>{\langle, \rangle_{BM, v}}>> \mathrm{Br}(k_v) \\
            @V{ev} V  V   @ AA \theta A  @ VV{id}V \\
       \mathrm{H}^1(k_v,  \Upsilon)  @. \times  \  \ \ \ \mathrm{H}^1(k_v, \widehat{\Upsilon})      @>{\cup}>> \mathrm{Br}(k_v) \\
      @A{\gamma}AA   @ VVV @ VV{-id}V \\
     \mathrm{H}^0 (k_v,  \Bbb G_m) \ @. \times \ \ \ \ \mathrm{H}^2(k_v, \Bbb Z)      @>\cup>> \mathrm{Br}(k_v)
   \end{CD} $$
 commutes by  \cite[Proposition 2.9]{ref5} and \cite[Chapter I, (1.4.7) Proposition]{ref6}, one obtains 
 $$ \langle g_v(K\otimes_k k_v)^\times, b \rangle_{BM, v}+ inv_v(\alpha) = - \chi (\rho_v(N_v(g_v))) $$
by \cite[(7.2.12) Proposition]{ref6} for all $v\in \Omega_k$ where $$\rho_v: k_v^\times \longrightarrow \Gal(k_v^{ab}/k_v)$$ is the local Artin map and $k_v^{ab}$ is the maximal abelian extension of $k_v$.  Therefore (\ref{iff}) is equivalent to that there is 
$$(g_v)_v \in ( \prod_{v\in \infty_k} X_{K, \mathcal A}(k_v) \times \prod_{v< \infty_k} C_v) \ \ \ \text{such that} \ \ \ \chi(\rho (N_v(g_v))_v)=0 $$
for all $\chi\in \mathrm{Hom} (\Gal((k^{ab} \cap K)/k), \mathbb Q/\mathbb Z)$ by Proposition \ref{br}. Therefore one concludes that $\rho ((N_v(g_v))$ is trivial in $\Gal((k^{ab}\cap K)/k)$ as desired.   This is equivalent to that 
$$ (N_v(g_v))_v \in k^\times \cdot N_{K/k} (\mathbb I_K)  $$
by  \cite[Chapter XIII, \S 9, Theorem 7]{Weil}. 
\end{proof}

\begin{coro} \label{max} Under the same assumption of Theorem \ref{rec},  an $\mathfrak o_k$-order  $\Xi$ of $K$ can be embedded into an $\mathfrak o_k$-order $\Gamma$ of $\mathcal A$ if and only if  there is 
$(g_v)_v\in G_{\mathcal A}(\mathbb A_k)$ with $$\Xi \otimes_{\mathfrak o_{k}} \mathfrak o_{k_v}  \subset g_v^{-1} (\Gamma \otimes_{\mathfrak o_k} \mathfrak o_{k_v} ) g_v  \ \ \ \text{ for all $v< \infty_k$ } $$ such that $\rho ((N_v(g_v))_v)$ is trivial in $\Gal((k(\Gamma)\cap K)/k)$. 
\end{coro} 

\begin{proof} Since $\Xi$ can be embedded into $\Gamma$ is equivalent to that $\Xi$ can be embedded into any element in $cls(\Gamma)$, one can replace $(g_v)_v$ with any element in the double cosets $G_{\mathcal A}(k)\cdot (g_v)_v \cdot St_{G_{\mathcal A}}(\Gamma)$ by (\ref{st}).  Therefore 
$$ (N_v(g_v))_v \in k^\times \cdot N_{K/k} (\mathbb I_K)   \ \ \ \Leftrightarrow \ \ \ (N_v(g_v))_v \in k^\times \cdot N_{K/k} (\mathbb I_K) \cdot N(St_{G_{\mathcal A}}(\Gamma)) $$
and the result follows from \cite[Chapter XIII, \S 9, Corollary 2]{Weil}. 
\end{proof}

We apply  Theorem \ref{rec} and Corollary \ref{max} to generalize \cite[Theorem 3.9]{LS1} as follows.

\begin{coro}\label{low-bound} Let $\mathcal A$ be a central simple algebra of degree $n$ over a number field $k$ satisfying the Eichler condition and $K/k$ be a field extension of degree $n$ inside $\mathcal A$. Let $\Xi$  be an $\mathfrak o_k$-order of $K$ and $\Gamma$ be an $\mathfrak o_k$-order of $\mathcal A$. 

If $\Xi\otimes_{\mathfrak o_k} \mathfrak o_{k_v}$ can be embedded into $\Gamma\otimes_{\mathfrak o_k} \mathfrak o_{k_v}$ as $\mathfrak o_{k_v}$-algebras for all $v< \infty_k$, then there are at least $[k(\Gamma): K\cap k(\Gamma)]$ classes in $gen(\Gamma)$ where $\Xi$ can be embedded as $\mathfrak o_k$-algebras. 
\end{coro}
\begin{proof}  Since $\Xi\otimes_{\mathfrak o_k} \mathfrak o_{k_v}$ can be embedded into $\Gamma\otimes_{\mathfrak o_k} \mathfrak o_{k_v}$ as $\mathfrak o_{k_v}$-algebras for all $v< \infty_k$, there is $\Gamma_0\in gen(\Gamma)$ such that $\Xi \subset \Gamma_0$. 
If a class in $gen(\Gamma)=gen(\Gamma_0)$ corresponding the double coset $G_{\mathcal A}(k) \cdot (g_v)_{v} \cdot St_{G_{\mathcal A}}(\Gamma_0)$ satisfies
$$ (N_v(g_v))_v \in k^\times \cdot N( St_{G_{\mathcal A}}(\Gamma_0)) \cdot N_{K/k} (\mathbb I_K) ,  $$
then $\Xi$ can be embedded into this class by Corollary \ref{max}. Therefore $\Xi$ can be embedded into the classes in $gen(\Gamma_0)$ corresponding to the cosets 
$$ k^\times \cdot N( St_{G_{\mathcal A}}(\Gamma_0)) \cdot N_{K/k} (\mathbb I_K)  / (k^\times \cdot N(St_{G_{\mathcal A}}(\Gamma_0)) $$
by Proposition \ref{parametrized}. Since
$$[k^\times \cdot N( St_{G_{\mathcal A}}(\Gamma_0)) \cdot N_{K/k} (\mathbb I_K)  : (k^\times \cdot N(St_{G_{\mathcal A}}(\Gamma_0)) ]= [k(\Gamma_0):k]/[k(\Gamma_0)\cap K:k] = [k(\Gamma_0): k(\Gamma_0)\cap K] ,$$ the result follows from $k(\Gamma)=k(\Gamma_0)$ as desired. 
\end{proof} 

We can also generalize \cite[Theorem 4.3]{LS1} as follows by using Corollary \ref{max}. 

\begin{coro}\label{tot-ram} Let $\mathcal A$ be a central simple algebra of degree $n$ over a number field $k$ satisfying the Eichler condition and $K/k$ be a field extension of degree $n$ inside $\mathcal A$.  Let $\Xi$  be an $\mathfrak o_k$-order of $K$ and $\Gamma$ be an $\mathfrak o_k$-order of $\mathcal A$. Assume that there is a finite prime $v_0$ of $k$ such that $\mathcal A\otimes_k k_{v_0}$ is a division algebra and $\Gamma\otimes_{\mathfrak o_k} \mathfrak o_{k_{v_0}}$ is maximal. 

If $\Xi\otimes_{\mathfrak o_k} \mathfrak o_{k_v}$ can be embedded into $\Gamma\otimes_{\mathfrak o_k} \mathfrak o_{k_v}$ as $\mathfrak o_{k_v}$-algebras for all $v< \infty_k$, then $\Xi$ can be embedded into every class of $gen(\Gamma)$ as $\mathfrak o_k$-algebras.
\end{coro} 

\begin{proof}
Since  $\mathcal A\otimes_k k_{v_0}$ is a division algebra and $\Gamma\otimes_{\mathfrak o_k} \mathfrak o_{k_{v_0}}$ is maximal,  the prime $v_0$ splits completely in $k(\Gamma)/k$ by Proposition \ref{property}. 2).  
On the other hand, there is only one prime $\frak V_0$ of $K$ above $v_0$ such that $[K_{\frak V_0}: k_{v_0}]=[K:k]$ by \cite[Proposition A.1]{PR}. This implies that $K\cap k(\Gamma)=k$.  This implies that $\Xi$ can be embedded into every class in $gen(\Gamma)$ by Corollary \ref{max}.
\end{proof}

Finally we can recover \cite[Th\'eor\`eme 10]{Ch} and \cite[Theorem 1]{Are}.

\begin{coro} \label{ratio} (Chevalley and Arenas-Carmona) Let $\mathcal A$ be a central simple algebra of degree $n$ over a number field $k$ satisfying the Eichler condition and $K/k$ be a field extension of degree $n$ inside $\mathcal A$.   If $\mathcal A\otimes_k k_v$ is either matrix algebras or division algebra for each $v<\infty_k$, then 
$$ \frac{\sharp \{\text{conjugacy classes of maximal orders in $\mathcal A$ where $\mathfrak{o}_K$ can be embedded} \} }{\sharp \{\text{conjugacy classes of maximal orders in $\mathcal A$} \}}=\frac{1}{[(K\cap k(\mathcal A)) : k]} . $$
\end{coro}

\begin{proof} Since there is a maximal order $\mathfrak o_{\mathcal A}$ containing $\mathfrak o_K$ by \cite[(10.4) Corollary]{Rein}, all conjugacy classes of maximal orders of $\mathcal A$ are parametrized by the double cosets $G_{\mathcal A}(k) \backslash G_{\mathcal A}(\mathbb A_k)/St_{G_{\mathcal A}}(\mathfrak o_{\mathcal A})$ by (\ref{st}).  By Corollary \ref{low-bound},  one only needs to show that the conjugacy class of maximal orders of $\mathcal A$ corresponding the double coset $G_{\mathcal A}(k) \cdot (\sigma_v)_{v} \cdot St_{G_{\mathcal A}}(\mathfrak o_{\mathcal A})$ where $\mathfrak{o}_K$ can be embedded satisfies 
\begin{equation} \label{non-empty} (N_v(\sigma_v))_v \in k^\times \cdot N( St_{G_{\mathcal A}}(\mathfrak o_{\mathcal A})) \cdot N_{K/k} (\mathbb I_K) . \end{equation}
Indeed,  there is $(g_v)_v \in G_{\mathcal A}(\mathbb A_k)$ with
$$ \mathfrak o_K \otimes_{\mathfrak o_{k}} \mathfrak o_{k_v}  \subset g_v^{-1} (\sigma_v (\mathfrak o_{\mathcal A}\otimes_{\mathfrak o_k} \mathfrak o_{k_v}) \sigma_v^{-1}) g_v $$ for all $v< \infty_k$ such that  
 $ (N_v(g_v))_v \in k^\times \cdot N_{K/k} (\mathbb I_K) $
by Theorem \ref{rec}. Therefore 
$$ g_v \cdot \sigma_v^{-1} \in \begin{cases}  ( \mathfrak o_{\mathcal A} \otimes_{\mathfrak o_k} \mathfrak o_{k_v})^\times \cdot (K\otimes_k k_v)^\times  \ \ \ & \text{if  $\mathcal A\otimes_k k_v \cong M_n(k_v)$} \\
(\mathcal A\otimes_k k_v)^\times \ \ \ & \text{if $\mathcal A\otimes_{k} k_v$ is a division algebra} \end{cases} $$
for all $v<\infty_k$ by Lemma \ref{local-orbit}.  This implies that $$ (g_v)_v  \cdot (\sigma_v)_v^{-1} \in St_{G_{\mathcal A}}(\mathfrak o_{\mathcal A}) \cdot \Bbb I_K $$
by Lemma \ref{stab}. Considering the reduced norms of the above element, one concludes that (\ref{non-empty}) follows from \cite[\S 9, Exercise 1]{Rein}. 
\end{proof} 

\begin{remark} In \cite{Ch}, the ratio in Corollary \ref{ratio} is given by $[(K\cap H_k): k]^{-1}$ where $H_k$ is the Hilbert class field of $k$. In fact, 
since $K/k$ is a finite extension of degree $n$ in Corollary \ref{ratio}, one has $$\mathbb I_k^n \subset N_{K/k} (\mathbb I_K) \ \ \ \text{and} \ \ \ 
 N_{K/k} (\mathbb I_K)  \cdot   (\prod_{v\in \infty_k} k_v^\times \times \prod_{v< \infty_k}  \mathfrak o_{k_v}^\times ) = N_{K/k} (\mathbb I_K)  \cdot   (\prod_{v\in \infty_k} k_v^\times \times \prod_{v< \infty_k}  \mathfrak o_{k_v}^\times ) \cdot \mathbb I_k^n . $$
This implies that $(K\cap H_k) = (K\cap k(\mathcal A))$ (see also Proposition \ref{property}). 
 \end{remark}

\addcontentsline{toc}{section}{

\end{document}